\documentclass[12pt]{amsart}

\usepackage[top=3.4cm,bottom=3.4cm,inner=3.5cm,outer=3.5cm]{geometry}
\usepackage{tikz-cd}
\usepackage{xcolor}
\usepackage{comment}

\newcommand{\disp}{\displaystyle}
\newcommand{\nc}{\newcommand}
\nc{\G}{{\Gamma}} \nc{\BC}{{\mathbb C}} \nc{\BQ}{{\mathbb Q}}
\nc{\BR}{{\mathbb R}} \nc{\BZ}{{\mathbb Z}} \nc{\BP}{{\mathbb P}} \nc{\PC}{{\BP_1(\BC)}}
\nc{\BN}{{\mathbb N}} \nc{\BM}{{\mathbb M}}
\nc{\fH}{{\mathbb H}}

\nc{\mat}{{\binom{a\,\ b}{c\,\ d}}}
\nc{\U}{{\mathcal U}}
\nc{\PS}{{\mbox{PSL}_2(\BZ)}} \nc{\SL}{{\mbox{SL}_2(\BZ)}}
\nc{\SR}{{\mbox{SL}_2(\BR)}} \nc{\PR}{{\mbox{PSL}_2(\BR)}}
\nc{\SLC}{{\mbox{SL}_2(\BC)}}

\nc{\GL}{{\mbox{GL}}} \nc{\PQ}{{\mbox{PGL}_2^+(\BQ)}}
\nc{\GR}{{\mbox{GL}_2^+(\BR)}} \nc{\PG}{{\mbox{PGL}_2(\BC)}}
\nc{\GC}{{\mbox{GL}_2(\BC)}}
\nc{\f}{{\mathcal{F}(\fH)}}
\nc{\Cc}{\widehat{\BC}}
\nc{\e}{{E_{\rho}(\G)}}
\nc{\g}{{\gamma}}
\nc{\vm}{{V_{\rho}(\G)}}
\nc{\oo}{{\mathcal O}}
\nc{\M}{{\mbox{M}}}
\nc{\om}{{\omega}}
\nc{\Om}{{\Omega}}
\nc{\TX}{{\widetilde{X}}}
\nc{\ol}{\overline}
\nc{\cl}{{\mathcal L}}
\nc{\ce}{{\mathcal E}}
\nc{\la}{{\lambda}}
\nc{\La}{{\Lambda}}

\nc{\cz}{{\mathcal Z}}

\newtheorem{numbered}{}[section]
\newtheorem{thm}[numbered]{Theorem}

\newtheorem{lem}[numbered]{Lemma}

\numberwithin{equation}{section}

\newcommand{\thmref}[1]{Theorem~\ref{#1}}

\newcommand{\lemref}[1]{Lemma~\ref{#1}}

\begin{document}

\title[]{On the modularity of solutions of certain differential equations of hypergeometric type }
\author[]{Hicham Saber} \author[]{Abdellah Sebbar}
\address{Department of Mathematics, Faculty of Science, University of Ha'il,   Ha'il, Kingdom of Saudi Arabia}
\address{Department of Mathematics and Statistics, University of Ottawa,
	 Ottawa Ontario K1N 6N5 Canada}
\email{hicham.saber7@gmail.com}
\email{asebbar@uottawa.ca}
\maketitle
\maketitle
\begin{abstract}
	The purpose of this paper is to provide  answers to some questions raised in a paper by Kaneko and Koike about the modularity of the solutions of a differential equations of hypergeometric type. In particular, we provide a number-theoretic explanation of why the modularity  of the solutions occurs in some cases and does not occur in other cases. This also proves their conjecture on the completeness of the list of modular solutions after adding some missing cases.
\end{abstract}

\section{Introduction}

This paper deals with some questions and conjectures raised in the paper \cite{kaneko-koike} regarding the solutions of the equation
\[
f''- \frac{2\pi i(k+1)}{6}E_2(\tau)\,f' \,+\,   \frac{2\pi i k (k+1)}{12}E'_2(\tau)\,f=0,
\]
where  $k$ is a rational number and $E_2$ is the weight 2 (quasi-modular) Eisenstein series. This differential equation appeared first in \cite{kaneko-zagier} for integers $k\equiv 0,\,4\mod 6$ in connection with the lifting of supersingular $j-$invariants of elliptic curves. In \cite{kaneko-koike}, modular solutions of the above differential equation are described explicitly when $k$ is an integer or half an integer. In addition, the authors conjectured, based on numerical experiments, that their solutions exhaust all the modular solutions. In the meantime, the normal form of the above differential equation can be shown to be
\begin{equation}\label{ss1}
	y''+\pi^2\left(\frac{k+1}{6}\right)^2E_4\,y=0,
\end{equation}
where $E_4$ is the weight 4 Eisenstein series. This Fuchsian differential equation has been the subject of our study in a more general setting by seeking the solutions to the equation $y''+\pi^2r^2E_4y=0$, where $r$ is a rational number. Our investigation relies on solving the Schwarzian equation $\{h,\tau\}=2r^2\pi^2E_4$ where $\{h,\tau\}$ is the Schwarz derivative of $h$. To be more precise, it is shown in \cite{forum}  that the Schwarzian equation admits solutions that are modular functions if and only if $r=n/m$ where $2\leq m\leq 5$ and $\gcd(m,n)=1$. In this case, the invariance group of the solution is $\G(m)$. In \cite{ramanujan}, we explicitly solved the same Schwarzian equation for each $r=k/6$ where $k\equiv 1 \mod 12$. Each solution is obtained from a  system of algebraic equations. Finally, in \cite{preprintSS}, the same equation is solved for every positive integer $r$ by solving a certain Eigenvector problem. Using these results,  we shall provide in this paper, among other results, a number-theoretic explanation for the modularity or non-modularity of the solutions. We also show that the cases for which all the solutions are modular is not exhaustive, contrary to what was conjectured in  \cite{kaneko-koike}.

\section{Automorphic Schwarzian equations}\label{mde}

The Schwarz derivative of a meromorphic function $h$ on the upper half-plane $\fH=\{\tau\in\BC | \Im \tau>0\}$ is defined by
\[
\{h,\tau\}=\left(\frac{h''}{h'}\right)'-\frac{1}{2}\left(\frac{h''}{h'}\right)^2.
\]
It is projectively invariant and $\{h,\tau\}=\{g,\tau\}$ if and only if $f$ is a linear fraction of $g$. Moreover, $\{h,\tau\}=0$ if and only if $h$ is a linear fraction of $\tau$. There is a close connection between the Schwarz derivative and second order ordinary differential equations. Indeed, let $R(\tau)$ be a meromorphic function on $\fH$ and let $y_1$ and $y_2$ be two linearly independent local solutions to $y''+R(\tau)\,y=0$, then locally, $h=y_2/y_1$ is a solution to the Schwarzian differential equation $\{h,\tau\}=2R(\tau)$. This connection yields most analytic properties of the Schwarz derivative.

If $w$ is a function of $\tau$, then we have the cocycle condition
\[
\{h,\tau\}d\tau^2=\{h,w\}dw^2+\{w,\tau\}d\tau^2.
\]
As a consequence, if $h$ is an automorphic function for a discrete subgroup $\G$ of $\SR$, then $\{h,\tau\}$ is a weight 4 automorphic form for $\G$ that has a double pole where $h'$ vanishes or where $h$ has at least a double pole, and it is holomorphic everywhere else including at the cusps of $\G$ \cite{mathann}. Conversely, if we suppose that $\{h,\tau\}$ is a (meromorphic) automorphic form of weight 4 for $\G$, then according to \cite{forum}, there exists a 2-dimensional complex representation $\rho$ of $\bar{\G}$, the image of $\G$ is $\PR$, such that for all $\tau\in\fH$ and $\gamma\in\G$
\[
h(\gamma\cdot \tau)\,=\,\rho(\gamma)\cdot h(\tau),
\]
where the action on both sides is by linear fractional transformation.
We call such $h$ a $\rho-$equivariant functions. The representation $\rho$ can be lifted naturally to $\G$.
 If $\rho=1$ is constant, then $h$ is an automorphic function, while if $\rho=\mbox{Id}$, then $h$ commutes with the action of $\G$ and we simply say that $h$ is an equivariant function for $\G$. For an arbitrary Fuchsian group $\G$ and for an arbitrary representation $\rho$ of $\G$, $\rho-$equivariant functions always exist \cite{kyushu}. 

To complete the connection between Schwarzian equations and ordinary differential equations, we have the following
\begin{thm}\cite[Theorem 3.3]{forum}\label{sol-2}
	Suppose that $f$ is a weight $4$  automorphic form for $\G$ that is holomorphic on $\fH$.
	\begin{enumerate}
		\item  If $y_1$ and $y_2$ are two linearly independent holomorphic solutions to $ y''+\frac{1}{2}\,f\,y=0$ then $\displaystyle F=\binom{y_1}{y_2}$ is a weight $-1$ vector-valued automorphic form for some multiplier system $\rho$. Furthermore, $h=y_1/y_2$ is a $\rho$-equivariant function satisfying $\{h,\tau\}=f$.
		\item If $h$ is a solution to $\{h,\tau\}=f$, then $h'$ does not vanish on $\fH$ and $y_1=h/\sqrt{h'}$ and $y_2=1/\sqrt{h'}$ are two linearly independent holomorphic solutions to $ y''+\frac{1}{2}\,f\,y=0$.
	\end{enumerate}
\end{thm}
We also introduce the main theorem from \cite{forum} which we will rely on to assert the modularity or non-modularity of the solutions to the modular differential equations under study.
\begin{thm}\cite[Theorem 8.3]{forum}\label{thm-forum}
	The Schwarzian equation $\{h,\tau\}=2\pi^2r^2E_4$ admits a modular function as a solution 
 if	and only if	$r=n/m$ with co-prime $m$ and $n$  and $2\leq m\leq 5$, in which case the invariance group is $\G(m)$.
	\end{thm}

\section{ Modular solutions }
In this section we will give some preparatory results about the modularity of the differential equations under review.

 Let $h$ be a  $\rho$-equivariant function on $\fH$, that is
\begin{equation}\label{equiv}
  h(\gamma \tau)\,=\,\rho(\gamma)\; h(\tau)\,,\ \tau\in\fH\,,\ \gamma\in\SL\,,
\end{equation}

and let $\G$ be a finite index subgroup of $\SL$, then we have

 \begin{lem}\label{equi-mod} Let $h$ be a non-constant $\rho-$equivariant function where $\rho$ is a 2-dimensional representation of the modular group. If there exists   $k\in \BZ$ such that  $h$ is a weight $k$ modular form on a finite index modular subgroup $\G$, then $k=0$ and $\G \subseteq \ker(\rho)$.
 \end{lem}

 \begin{proof}
 	Suppose that $h$ is a weight $k$ modular form for a finite index subgroup $\G$. For $\disp  \gamma=\binom{a_{\gamma}\ \ b_{\gamma}}{c_{\gamma}\ \ d_{\gamma}}\in\G$, we have
 \[
 (c_{\gamma}\tau+d_{\gamma})^k\,h(\tau)=h(\gamma \tau)=\rho(\gamma)\,h(\tau)
 \]
In other words,
\begin{equation}\label{period}
 (c_{\gamma}\tau+d_{\gamma})^k\,=\,\frac{1}{h(\tau)}\rho(\gamma)h(\tau).
\end{equation}
Since $\G$ is a finite index modular subgroup, there exists $\g\in\G$ such that $c_{\gamma}\neq 0$. Also,
let $n_{\G}$ be the cusp width at $\infty$ for $\G$, that is the least positive integer such that $\binom{1\ \ n_{\G}}{0\ \ \ 1}\in\G$. Then for all $\tau\in\fH$, $h(\tau+n_{\G})=h(\tau)$. Therefore, the right hand side of \eqref{period} is $n_{\G}-$periodic while the right hand side is not unless $k=0$. Since $h$ is a non-constant meromorphic function on $\fH$, it assumes infinitely many values, and clearly $\G \subseteq \ker(\rho)$.
 \end{proof}

We now turn our attention to the main differential equation of hypergeometric type that was the object of study in \cite{kaneko-koike,kaneko-zagier}
\begin{equation}\label{kan-koi}
f''- \frac{2\pi i(k+1)}{6}E_2(\tau)\,f' \,+\,   \frac{2\pi i k (k+1)}{12}E'_2(\tau)\,f=0.
\end{equation}
Here $k$ is a non-negative integer and $E_2$ and $E_4$ are the Eisenstein series given  by their $q-$series where $q=\exp(2\pi i\tau)$:
\begin{align*}
	E_2(\tau)&=1-24\,\sum_{n\geq 1}\,\sigma_1(n)\,q^n\,,\\
	E_4(\tau)&=1+240\,\sum_{n\geq 1}\,\sigma_3(n)\,q^n\,.
\end{align*}
 Through the change of function $\displaystyle f= \eta^{2(k+1)}\, y $, where $\eta$ is
the Dedekind eta-function defined by
\[
\eta(\tau)\,=\,q^{\frac{1}{24}}\,\prod_{k\geq1}(1-q^n)\,,\ \
\]
and using the fact that $\disp E_2=\frac{6}{i\pi}\frac{\eta'}{\eta}$ and that $\displaystyle \frac{6}{i \pi }E'_2 = E_2^2-E_4$, we see that
 the equation \eqref{kan-koi} is equivalent to
\begin{equation}\label{ASE}
y''+  \pi^2 \left(\frac{k+1}{6}\right)^2 E_4(\tau)\,y=0.
\end{equation}
\begin{thm}
  The differential equation (\ref{kan-koi}) has two linearly independent modular solutions, for some finite index subgroup $\G$ of $\SL$, if and only if there exist  two positive integers $n$ and $m$  satisfying $2\leq m\leq 5$, $\gcd(m,n)=1$, and
\[
 \frac{k+1}{6}=  \frac{n}{m}.
 \]
\end{thm}
\begin{proof}
  If the differential equation (\ref{kan-koi}) has two linearly independent solutions, $g_1$ and $g_2$, which are  modular for some finite index subgroup $\G$ of $\SL$, then  $\displaystyle h=g_2/g_1$ will be a solution to the Shcwarzian equation
\[
\{h,\tau\}\,=\,8 \pi^2 \left(\frac{k+1}{12}\right)^2\,E_4(\tau).
\]
Therefore, $h$ is $\rho$-equivariant for some 2-dimensional representation $\rho$ of the modular group. Since $h$ is also a  modular form for $\G$, then by  \lemref{equi-mod}, $h$ must be a modular function for $\G$. In particular, $g_1$ and $g_2$ have the same weight.
By \thmref{thm-forum}, we must have
 \[
 8 \pi^2 \left(\frac{k+1}{12}\right)^2= 2 \pi^2 \left(\frac{n}{m}\right)^2
 \]
with $m$ and $n$ two positive integers satisfying $2\leq m\leq 5$ and $\gcd(m,n)=1$, which implies that
\[
 \frac{k+1}{6}=  \frac{n}{m}
 \]
 as desired.

 For the converse, by  \thmref{thm-forum}, the equation
 \[
\{h,\tau\}\,=\,2 \pi^2 \left(\frac{k+1}{6}\right)^2\,E_4(\tau),
\]
 has a $\G(m)$ modular function solution $h$, and by \thmref{sol-2}, we see that  $y_1=h/\sqrt{h'}$ and $y_2=1/\sqrt{h'}$ are two linearly independent holomorphic  solutions of
\[
y''+  \pi^2 \left(\frac{k+1}{6}\right)^2 \,E_4(\tau)\,y=0.
\]
It is clear that $y_1$ and $y_2$ are modular forms of weight $-1$ and a character $\chi$ for $\G(m)$, with  $\chi^2=1$, hence they are also modular for
the finite index group $\G=\mbox{ker}(\chi)$. The required solutions  are then  $\displaystyle f_1= \eta^{2(k+1)}\, y_1 $   and
$\displaystyle f_2= \eta^{2(k+1)}\, y_2 $ which are both of weight $k$.
\end{proof}

\section{A second look at the solutions obtained by Kaneko and Koike}
We now proceed to answer the question raised at the end of Section 3 of \cite{kaneko-koike}. In particular, we provide a number-theoretic explanation  for the modularity or non-modularity of the solutions to \eqref{kan-koi} that appeared in Theorem 1 and also we explain the quasi-modularity of the the solution in Theorem 2 of \cite{kaneko-koike}. Furthermore we show that these solutions do not exhaust all the modular solutions as it was conjectured at the end of the paper.


 When $k\equiv 1,\, 2,\, 3\, \mod 6$, or when $k$ is a half integer congruent to $1/2$ modulo $3$, then $(k+1)/6$ is a rational number whose reduced form has a denominator between 2 and 4. Therefore,
  \eqref{kan-koi}  has two linearly independent solutions, say $f_1$ and $f_2$, modular for certain congruence groups $\G_1$ and $\G_2$ respectively. Therefore any other solution $\displaystyle f= af_1+bf_2$ is  at least modular for $\G_1 \cap \G_2$. This explains the modularity  of the solutions in Theorem 1. of \cite{kaneko-koike}.

 In case $ \displaystyle k\equiv 0,\, 4 \mod 6$,  there  exists a one-dimensional modular solution  space, say generated by $g_1$. If $g_2$ is another solution that is linearly independent with $g_1$, then $g_2$ cannot be modular, otherwise, we  will have
\[
 \frac{k+1}{6}=  \frac{n}{m},
 \]
with $m$ and $n$ two positive integers satisfying $2\leq m\leq 5$ and $\gcd(m,n)=1$. Thus $ \displaystyle m(k+1)=6 \, n $. As $k+1\equiv \pm 1\mod 6$, we must have $m\equiv 0\mod 6$ which is impossible since $2\leq m\leq 5$.


 The case $ \displaystyle k\equiv \, 5 \, [6]$ requires a different approach since $\frac{k+1}{6}$ is now an integer. According to \cite{preprintSS}, we have
 \begin{thm}\cite[Theorem 4.2]{preprintSS} Let $r$ be a positive integer and let $\G=\SL$ if $r$ is even and $\G=\SL^2$ if $r$ is odd.
 	Then there exist two linearly independent solutions $y_1$ and $y_2$ to
 	$y''+\pi^2r^2E_4y=0$  such that:
 	\begin{enumerate}
 		\item $y_1$ and $y_2$ have the following $q-$expansions
 		\[
 		y_1(\tau)=q^{r/2}\,\sum_{n\geq 0}\, \alpha_n q^n\, \ \alpha_0\neq 0
 		\]
 		and
 		\[
 		y_2(\tau)=\tau y_1(\tau)+q^{-r/2}\,\sum_{n\geq 0}\, \beta_n q^n\,,\  \beta_0\neq 0.
 		\]
 		\item $y_1$ is a quasi-modular form  of weight zero and depth one for $\Gamma$ and $y_2(\tau)-\tau y_1(\tau)$ is a modular form of weight $-2$ for $\Gamma$.
 		\item $h_r=y_2/y_1$ is equivariant for $\G$.
 	\end{enumerate}	
 \end{thm}
 Here, $\SL^2$ is the subgroup made of the squares of elements of $\SL$, and it is the unique subgroup of index 2 in $\SL$.

 It follows that if $k\equiv 5\mod 6$, then we cannot have modular solutions. However, $f=\eta^{2k+2}y_1$ is a quasi-modular solution to \eqref{kan-koi} that has weight $k+1$ and depth 1, which is in line with  what is claimed in \cite{kaneko-koike}.

Thus, we have provided a number theoretic justification why certain solutions to \eqref{kan-koi} are modular of weight $k$, others are quasi-modular of weight $k+1$, and some are neither. However, while these claims hold when $k$ is an integer or a half integer congruent to $1/2$ modulo $3$, the list is not exhaustive as it was conjectured in \cite{kaneko-koike}. Indeed, what is missing and would make the list complete are the level 5 solutions. One needs to include the rational numbers $k$ such that $\disp \frac{k+1}{6}=\frac{n}{5}$ for $n$ not divisible by 5, that is,  $k=6n/5 -1$. For example, for $k=1/5$, the Hauptmodul for $\G(5)$ given by
\[
t=q^{\frac{1}{5}}\prod_{n\geq1}\,(1-q^n)^{\left(\frac{n}{5}\right)},
\]
where $\left(\frac{\,\cdot\,}{\,\cdot\,}\right)$ is the Legendre symbol, is a solution to $\disp \{h,\tau\}=2(1/5)^2\pi^2E_4$, \cite{forum}.  Thus, the modular solutions to \eqref{kan-koi} are $f_1=\eta^{2k+1}t'^{-1/2}$ and $f_2=\eta^{2k+1}tt'^{-1/2}$ and both are of weight 1/5. In the cases $k=7/5$
and $13/5$,
the solution to the Schwarzian equation are given respectively by
\[
\frac{t^2(t^5-7)}{7t^5+1}
\]
and
\[
\frac{t^3(t^{10}-39t^5-26)}{26t^{10}-39t^5-1}
\]
from which one can write down the modular solutions to \eqref{kan-koi}.


\end{document}